\documentclass[12pt]{article}
\usepackage{latexsym}
\usepackage{amsmath,amsfonts,amssymb,mathrsfs,graphicx}
\usepackage{amsthm}
\usepackage{graphicx}
\usepackage{color}
\usepackage{mathtools}
\usepackage{verbatim}
\title{Fractal percolation on statistically self-affine carpets}
\author{Kenneth Falconer and Tianyi Feng\\
\small{{\it Mathematical Institute,  
University of St~Andrews, North Haugh, St~Andrews,}} \\
\small{{\it Fife, KY16~9SS, Scotland }}} 
\date{}

\newcommand{\be}{\begin{equation}}
\newcommand{\ee}{\end{equation}}

\renewcommand{\P}{\mathbb{P}}

\newtheorem{theorem}{Theorem}[section]
\newtheorem{lemma}[theorem]{Lemma}

\newtheorem{ques}[theorem]{Question}
\newtheorem{prop}[theorem]{Proposition}
\theoremstyle{definition}

\theoremstyle{remark}

\numberwithin{equation}{section}

\begin{document}
\maketitle
\begin{abstract}
We consider a random self-affine carpet $F$ based on an $n\times m$ subdivision of rectangles and a probability $0<p<1$.  Starting by dividing $[0,1]^2$  into an $n\times m$ grid of rectangles and selecting these independently with probability $p$, we then divide the selected rectangles into $n\times m$ subrectangles which are again selected with probability $p$; we continue in this way to obtain a statistically self-affine set $F$.
We are particularly interested in topological properties of $F$. We show that the critical value of $p$ above which there is a positive probability that $F$ connects the left and right edges of $[0,1]^2$ is the same as the critical value for  $F$ to connect the top and bottom edges of $[0,1]^2$.
Once this is established we derive various topological properties of $F$ analogous to those known for self-similar carpets.
\end{abstract}

\section{Introduction}\label{setting}
The notion of fractal percolation was introduced in 1974 by Mandelbrot  \cite{Man} and discussed in his classic book \cite{Manbook}, where it is termed `canonical curdling'. Fractal percolation concerns the topological properties of a statistically self-similar set, obtained by repeating a random operation at ever-decreasing scales. 

The square-based fractal percolation model, which has been well-studied \cite{CCD, DM, Don, Mee}, is based on repeated subdivision of the unit square  into smaller squares which are selected at random.  Let $m\geq 2$ be an integer and   $0\leq p\leq 1$ a probability. The unit square $E_0=[0,1]^2$ is divided into a grid of $m^2$ closed subsquares each of side-length $m^{-1}$  and each of these subsquares is selected independently with probability $p$; we write $E_1$ for the random set formed as the union of these selected squares. We next divide each of these selected subsquares into a grid of squares of side $m^{-2}$, and, in a similar way, select these subsquares independently with probability $p$, with their union forming the closed set $E_2$. We continue in this way to obtain a decreasing sequence of closed sets $E_k$ where each $E_k$ is a union of squares of side $m^{-k}$. The random closed set $F:=\bigcap_{k=0}^\infty E_k$ is termed a {\it statistically self-similar carpet}. For $p$ sufficiently large there is a positive probability that $F$ is non-empty and, conditional on this, various properties of $F$ have been examined, including the  almost sure box-counting and Hausdorff dimensions of $F$. Of particular interest here are the topological properties of $F$ which change dramatically as $p$ increases through a {\it critical value} $p_C$. For  $p<p_C$ the random set $F$  is almost surely totally disconnected, whereas if $p\geq p_C$  it almost surely has many non-trivial connected components and there is a positive probability of the left and right sides of the square being joined through $F$, and similarly for the top and bottom sides. Finding the value of $p_C$ even for small $m$ is difficult, but some (not very close) lower and upper bounds have been obtained, see \cite{CCD, Don, Whi}. 

A number of works examine the structure of the random set $F$ itself. For example, \cite{BJJKP} shows that there is an $\alpha< 1$ such that $F$ is almost surely purely $\alpha$-unrectifiable, that is intersects every $\alpha$-H\"{o}lder curve in a set of $1/\alpha$-dimensional Hausdorff measure 0. Various generalizations of this model have been investigated, for example \cite{BC} considers a more general class of continuum fractal percolation showing that connected components at the critical probability is a consequence of scale invariance. Then \cite{BEMT} considers phase transitions in a fractal cylinder process,  a fractal version of the Poisson cylider model.

Rather more general statistically self-similar sets have been considered, particularly from a dimensional viewpoint, see for example \cite{Fal, MW, Tro1}. This paper is concerned with statistically self-affine carpets. These are constructed, for integers $m>n\geq 2$ and a  probability $0\leq p\leq 1$, in a similar way to the self-similar square-based model above, except we select subrectangles rather than subsquares. Thus $E_0=[0,1]^2$ is divided into an $n\times m$ array of rectangles of sides $n^{-1} \times m^{-1}$ which are selected independently with probability $p$, the selected rectangles being repeatedly subdivided and selected in the same way, to get a decreasing sequence of sets $E_k$ formed as a union of $n^{-k} \times m^{-k}$ rectangles, to yield a statistically self-affine set $F=\bigcap_{k=0}^\infty E_k$. A more precise description of this model will be given in the next section.

\begin{figure}[h]
\begin{center}
\includegraphics[width = 0.9\textwidth]{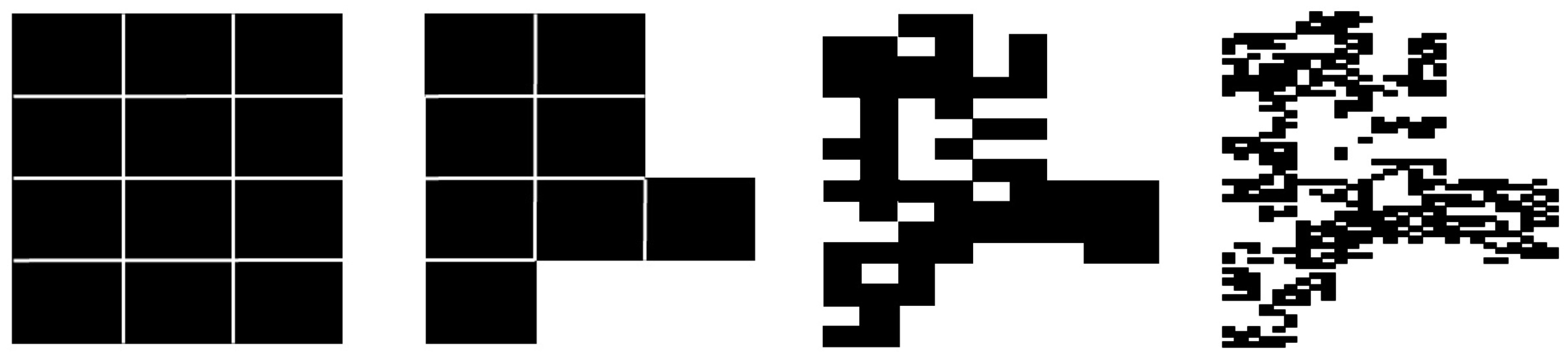}
$\qquad E_0  =[0,1]^2 \qquad\qquad\quad\quad E_1 \qquad\qquad\qquad\qquad  E_2 \qquad\qquad\qquad \quad\ E_3\!\qquad$
\caption{The first three steps in the construction of a statistically self-affine $F$.}\label{percEi}
\end{center}
\end{figure}

We are interested in topological properties of this statistically self-affine $F$ and in particular to what extent a random construction based on hierarchy of rectangles differs from the square case. With the rectangles becoming wide and squat for large $k$, it is not clear that the critical probability for horizontal (left-right) crossings to occur in $F$ is the same as that for vertical (top-bottom) crossings. We will show that these critical probabilities are in fact equal, that is if for some $p$ there is a positive probability of a horizontal crossing then there is also a positive probability of a vertical crossing, and vice-versa. Once this is established, modifying methods used for the statistically self-similar $m \times m$ case yields information about the topology of the random set $F$.

We have aimed to make this  article fairly self-contained by including some proofs that are reminiscent of arguments that have been used for percolation on statistically self-similar carpets.

\section{Model and results}\label{model}
Here we describe the construction of the random set $F$ and state our results.  
We fix integers $m>n\geq 2$ throughout the construction and its analysis.

We divide the unit square $E_0:= [0,1]^2$ into a rectangular array of  $mn$ level-1 closed rectangles of sides $n^{-1}\times m^{-1}$ which we label (in some order)  $\{R_{i_1}\}_{1\leq i_1 \leq mn}$. We then divide each  $\{R_{i_1}\}$  into $mn$ level-2 closed rectangles  $\{R_{i_1,i_2}\}_{1\leq i_2 \leq mn}$ of sides $n^{-2}\times m^{-2}$. We continue in this way, so that $\{R_{i_1,\ldots, i_k}\}_{1\leq i_1,\ldots, i_k \leq mn}$ are  $(mn)^k$ level-$k$ closed rectangles of sides $n^{-k}\times m^{-k}$, with $R_{i_1,\ldots, i_k}= \bigcup_{1\leq i \leq mn} R_{i_1,\ldots, i_k, i}$ for each $k$ and $(i_1,\ldots,i_k)$.

Let $0\leq p \leq 1$ be a probability.  We select each rectangle $\{R_{i_1}\}$ independently with probability $p$ and reject it with probability $1-p$ and we write $E_1$ for the union of these selected rectangles. For each of the selected rectangles $\{R_{i_1}\}$ we select  level-2 subrectangles $\{R_{i_1,i_2}\}$ with probability $p$ and let $E_2$ be the union of all these selected rectangles. We continue in this way to get a decreasing sequence of closed sets $E_k$, the union of a random selection of $n^{-k}\times m^{-k}$ closed rectangles, and let $F= \bigcap_{k=1}^\infty E_k$ be the random limit set which will be closed. The number of rectangles in $E_k$ is given by a Galton-Watson branching process which implies in particular that there is a positive probability of $F$ being non-empty if and only if $p>1/mn$.

We formalise this by defining a probability space $(\Omega, \mathcal{F},\P)$ on the underlying rectangles  by 
\begin{equation}\label{Ek}
\Omega ={\mathcal {P}}\Big(\big(\{R_{i_1}\}_{1\leq i_1 \leq mn},\{R_{i_1,i_2}\}_{1\leq i_1,i_2 \leq mn},\ldots\big)\Big)\end{equation}
where $\mathcal {P}$ denotes the power set, with $\mathcal{F}$ the product $\sigma$-field on this countable sequence of sets, and the product probability $\P$ defined by setting $\P( R_{i_1,\ldots,i_k})=p$ independently for all
$k$ and all  $(i_1,\ldots,i_k) \ (1\leq i_j\leq mn)$, and extending to $\mathcal{F}$. We identify $R_{i_1,\ldots,i_k}\in \omega$ with the event  `the rectangle $R_{i_1,\ldots,i_k}$ is selected' in the random geometric construction.

Then for $k\geq 1$,
$$E_k =  \bigcup \{ R_{i_1,\ldots,i_k}: R_{i_1}, R_{i_1,i_2},\ldots.,R_{i_1,\ldots,i_k}\in \omega\}=
\bigcap_{j=1}^k  \bigcup \{ R_{i_1,\ldots,i_j}\in \omega\} \quad  (\omega \in \Omega),$$
and the random self-affine set  $F$ obtained by this process is.
$$F \equiv F(\omega) = \bigcap_{k=1}^\infty E_k =  \bigcap_{j=1}^\infty  \bigcup \{ R_{i_1,\ldots,i_j}\in \omega\} \quad  (\omega \in \Omega)$$
Note that it is convenient to consider selection or rejection of all the rectangles in the space, even if their `parent' rectangle has not been selected.
We also let
$$F_k =  \bigcap_{j=k}^\infty  \bigcup \{ R_{i_1,\ldots,i_j}\in \omega\}\quad  (\omega \in \Omega);$$
thus $F_1= F$ and, for $k\geq 1$, $F_k$ is the random set obtained starting the random selection   at level $k$, and assuming that all rectangle at levels $1$ to $k-1$ are selected. Note that the sets $E_k$ and thus $F$ and  $F_k$ are compact.

We will be interested in topological properties and especially connectivity properties of $F$ and how they vary with $p$, see Figure \ref{pics}. In particular we consider when there is a positive probability of $F$ providing vertical and/or horizontal crossings of the unit square, that is when there are points in opposite sides of the square that are in the same connected component $F$. 
\begin{figure}[h]
\begin{center}
\includegraphics[width = 0.25\textwidth]{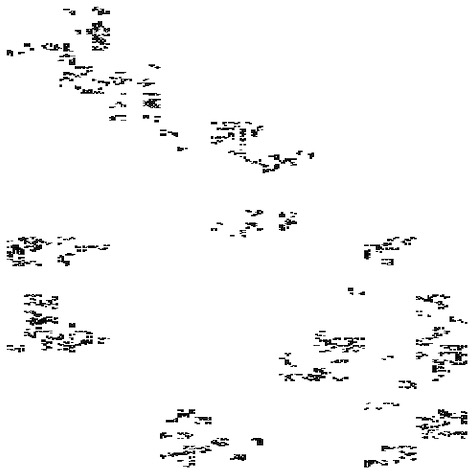}
\qquad
\includegraphics[width = 0.25\textwidth]{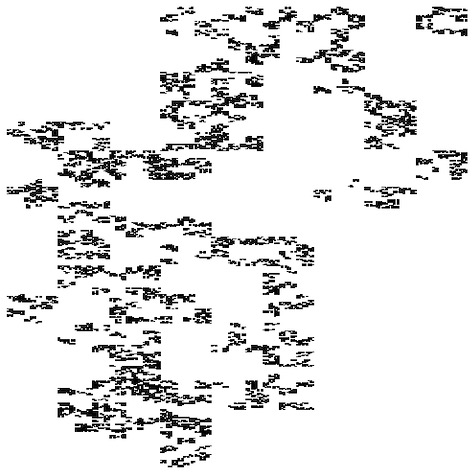}
\medskip
\qquad
\includegraphics[width = 0.25\textwidth]{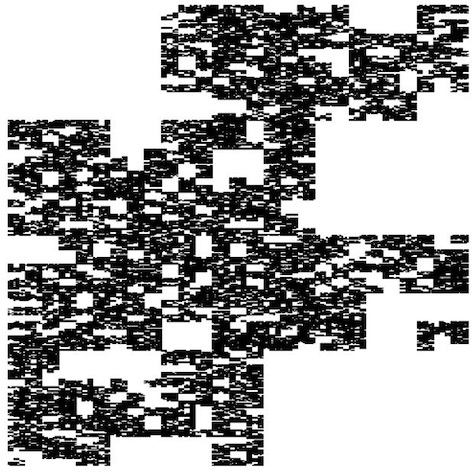}\\
\medskip 
(a) $p=0.4$\qquad \qquad \qquad \qquad (b) $p=0.5$
\qquad \qquad \qquad  
 (c) $p=0.78$\qquad 
\caption{Realisations of $F$ based on a $3\times 4$ rectangular grid, for different retention probabilities $p$.}\label{pics}
\end{center}
\end{figure}

We review first some basic properties of statistically self-affine sets. With $F$ constructed as above, the following proposition records the probability of $F$ being empty, and the almost sure box-counting, Hausdorff and Assouad dimensions of $F$ if it is non-empty. 

\begin{prop}\label{dims} 
$({\rm i})$ If $0\leq p\leq 1/mn$ then $F=\emptyset$ almost surely. If $1/mn<p \leq 1$ then $\P(F\neq \emptyset)=t>0$, where $t$ is the least non-negative solution of $t= (pt+1-p)^{mn}$.

$({\rm ii})$ If $1/mn<p \leq 1$ then, conditional on  $F\neq \emptyset$, 
$$\dim_H F = \dim_B F = 
\left\{
\begin{array}{ll}
 \log(pnm)\big/\log n &  (1/mn<p\leq 1/m)    \\
 \log(pm^2)\big/\log m &  (1/m<p\leq 1) 
   \end{array}
\right. ,
$$
$$\dim_A F = 2 \ \ (1/mn<p\leq 1).$$
\end{prop}

The proof of Proposition  \ref{dims} is sketched in Section \ref{basic}.

Let $F$ be a  subset of $\mathbb{R}^2$. We say that two points ${\bf x}_1,{\bf x}_2 \in F$ are {\it connected} or  {\it joined}, or  that ${\bf x}_1$ {\it is connected to} ${\bf x}_2$ in $F$  if they are in the same connected component of $F$ (with respect to the usual Euclidean topology). We say that  two sets $A,B \subset \mathbb{R}^2$ are {\it connected} or {\it joined} in $F$ if there are points ${\bf x}_1\in A$ and ${\bf x}_2\in B$ that are connected in $F$.

The set $F$ is a {\it horizontal} or H-{\it crossing} of the unit square $[0,1]^2$ if $\{0\}\times [0,1]$ is connected to $\{1\}\times [0,1]$. Similarly $F$ is a {\it vertical} or V-{\it crossing} if $ [0,1]\times \{0\}$ is connected to $ [0,1]\times \{1\}$. For a probability $0\leq p\leq 1$ we write $\theta_H(p)$ and $\theta_V(p)$ for the probabilities that $F$ has a horizontal or vertical crossing respectively. The following proposition, which follows by a similar argument to the statistically self-similar case, asserts that there is a positive probability of horizontal and vertical crossings if $p$ is sufficiently close to 1.

\begin{prop}\label{crit} 
There exist numbers $p_{H},p_{V}\in (0,1)$ such that:

 if $0\leq p <p_{H}$ then $\theta_H(p)= 0$ and if $p_{H}<p\leq 1$ then $\theta_H(p)> 0$, 

 if $0\leq p <p_{V}$ then $\theta_V(p)= 0$ and if $p_{V}<p\leq 1$ then $\theta_V(p)> 0$. 
\end{prop}

The numbers $p_{H}$ and $p_{V}$ in Proposition \ref{crit} are termed  {\it critical probabilities} for  H- and V-crossings respectively. Adding numerical details to the proof, given in Section \ref{basic},  gives  upper bounds for the critical probabilities that are extremely close close to 1 and which are far from optimal. The method may be refined to get better values, but even for percolation on a statistically self-similar set, the best-known bounds for the critical probabilities are poor, see \cite{CCD, Don, Whi}.

We will show that the horizontal and vertical critical probabilities are equal, from which other properties of $F$ will follow.

\begin{theorem}\label{thmbrief} Let $0\leq p\leq 1$. Then $\theta_H(p)>0$ if and only if $\theta_V(p)>0$. Equivalently  $p_{H}=p_{V}$.
\end{theorem}
Given this theorem we will write $p_C :=p_{H}=p_{V}$ for the common critical value.

We will prove Theorem \ref{thmbrief} in Section \ref{mainbrief}. We will then refine Proposition \ref{crit} to conclude that there is a strictly positive probability of both H- and V-crossings at $p=p_{C}$ where $\theta_H(p)$ and $\theta_V(p)$ undergo a jump phase transition.

\begin{theorem}\label{phase} 
The functions $\theta_H$ and $ \theta_V$ are right-continuous and non-decreasing on $[0,1]$.
If $0\leq p <p_{C}$ then $\theta_H(p)= \theta_V(p)=0$ and if $p_{C}\leq p\leq 1$ then $\theta_H(p), \theta_V(p)> 0$. 
\end{theorem}

Given Theorem \ref{phase}, one can deduce more about $\theta_H(p)$ and $\theta_V(p)$ as functions of $p$ as well as about the topological nature of $F$ using methods analogous to those that have been developed for percolation on statistically self-similar sets, in particular by Meester \cite{Mee}. 

\begin{theorem}\label{cty} 
\noindent (i) If $1/mn <p<p_{C}$ then, conditional on $F \neq \emptyset$, the set $F$ is  uncountable and totally disconnected. 

\noindent (ii) If $p_{C} \leq p<1$ then, conditional on $F \neq \emptyset$, the set $F$ includes a countably infinite number of disjoint non-trivial connected components.
\end{theorem}

In Section \ref{final} we indicate some other properties of $F$ and mention some open questions.

\section{Proofs}

This section contains proofs of the results stated above.

\subsection{Proofs of basic properties}\label{basic}

\noindent {\it Proof of Proposition \ref{dims}}
(i) Let $N_p(k)$ be the number of  rectangles in $E_k$ remaining at the $k$th stage of the construction of $F$. Then $N_p(k)$ is a Galton-Watson branching process with offspring expectation $pnm$. By standard branching process properties \cite{AN}, if $pnm\leq 1$ then almost surely the process becomes extinct  so $F=\emptyset$, whereas if  $pnm> 1$ then  $N_p(k)\not\to 0$ and indeed $N_p(k)\to \infty$, so that $F\neq \emptyset$, with probability $t$, the least positive number such that $t=\sum_{j=0}^{mn}\P(N_p(k)=j)t^j =\sum_{j=0}^{mn}\binom{mn}{j} p^j(1-p)^{mn-j}t^j=(pt+1-p)^{mn}$.

(ii) These dimensions may be gleaned from more general results on dimensions of statistically self-affine sets. For box-counting dimension see \cite[Equation (1.2)]{BF},  \cite[Theorem 4.1]{LG} or \cite[Corollary 4.6]{Tro}, and for Hausdorff dimension see \cite[Equation (1.2)]{BF} or \cite{Tro}. The calculations use that the projection of $F$ onto the horizontal axis has box and Hausdorff dimensions $\min\{ 1, \log (p n m)/\log n\}$ almost surely, subject to non-extinction. As with many random constructions, the Assouad dimension is full, see \cite[Theorem 3.2]{FMT}.
$\Box$ \hfill
\medskip

\noindent {\it Proof of Proposition \ref{crit}}
We sketch this proof which is very similar to \cite[Section 2]{CCD} or \cite[Section 4]{DM}. The probability of a crossing increases with $p$, and by part (i) is 0 if $0 <p<1/m$ since then $\dim_H F <1$ and $F$ is totally disconnected. Thus it is enough to show that there is some number $p_A\in (0,1)$ such that if $p_A\leq p \leq 1$ there are, with positive probability, both H- and V-crossings. Let $3\leq n <m$. We observe that if $R$ and $R'$ are edge-adjacent level-$k$ rectangles in $E_k$,  both containing either $mn-1$ or $mn$ level-$(k+1)$ selected subrectangles in   $E_{k+1}$, then these subrectangles form a connected unit.

We say that a rectangle of $E_k$ is {\it full} if it contains $mn-1$ or $mn$ subrectangles of   $E_{k+1}$. We say that a rectangle of $E_k$ is 2-{\it full} if it contains 
$mn-1$ or $mn$ full subrectangles of   $E_{k+1}$, and, inductively,  a rectangle of $E_k$ is $j$-{\it full} if it contains $mn-1$ or $mn$ subrectangles of  $E_{k+1}$ that are $(j-1)$-full. By the above observation, if $E_0 = [0,1]^2$ is $j$-full then opposite sides of the square $[0,1]^2$ are connected by a sequence of abutting rectangles in $E_j$.

Writing  $p_j$ for the probability that $E_0$ is $j$-full, by using self-affinity and the definition of $j$-full we obtain a polynomial recursion $p_j =f_p(p_{j-1})$ where $f_p(t) =mnp^{mn-1} t^{mn-1} -(mn-1)p^{mn}t^{mn}$.  A calculation similar to \cite[Section 2]{CCD} shows that if $p$ is sufficiently close to 1, say $p_A \leq p <1$, the polynomial $f_p$ has stable fixed point $t_p\in (0,1)$ with $p_j =f_p^j(1) \to t_p$. Thus there is a  positive probability at least $ t_p>0$ that $E_0$ is $j$-full for all $j\in \mathbb{N}$, in which case there are both H- and V-crossings in $E_j$ for all $j$. As  $E_j$ is a decreasing sequence of closed sets with $F= \bigcap_{j=1}^\infty E_j$, we conclude that $F$ has both H- and V-crossings with positive probability if $p\geq p_A$.

The same conclusion holds for the $2\times m$ random construction $(m\geq 3)$ with a suitable definition of `full'.
 \hfill $\Box$

\subsection{Proof of Theorem 2.3}\label{mainbrief}

The following simple ideas are key to the proofs.. 

$\bullet$ If an event of positive probability is a finite or countable union of sub-events, then at least one of the sub-events has positive probability.

$\bullet$ $F$ is {\it statistically self-affine}, that is for a level-$k$ rectangle $R$,  the random set $F_{k+1}\cap R$ has the same distribution as $\phi(F)$ (with the obvious translation to $R$) where $\phi$ is the affine map $\phi(x,y)= (n^{-k}x, m^{-k}y)$. Similarly, $F\cap R$ and $F\cap R'$ have identical distribution for level-$k$ rectangles $R$ and $R'$.

$\bullet$ Limiting events such as crossings occur with positive probability in $F$ if and only if  they occur with positive probability in $F_k$ for some, and thus for all, $k\in \mathbb{N}$ (this follows since there is a positive probability of selecting every level-$j$ rectangle for all $1\leq j\leq k-1$).

$\bullet$ Harris' inequality (or the FKG inequality), in the form that if $A,B$ are increasing events then $\P(A\cap B) \geq \P(A)\P(B)$, see \cite[Section 2.2]{Gri} or\cite{CCD}. Recall that an event $C$ is increasing if  $I_C(\omega) \leq I_C(\omega')$ whenever $\omega \subseteq \omega'$, where $I_C$ is the indicator function of $C$, that is if an event occurs for a selection of rectangles it also occurs for any larger selection. Many of the events we consider here, such as the occurrence of crossings, are increasing.

We first note that $F$  does not contain line segments parallel to the axes.

\begin{lemma}\label{seg}
Let $0\leq p<1$. Almost surely $F$ contains no horizontal and no vertical line segments. The same is true for $F_k$ for all $k$.
\end{lemma}
\begin{proof}
We show that $F$ contains no horizontal line segment, the vertical case is similar.

First consider a line segment $S$ of the form 
$$[an^{-j}, (a+1)n^{-j}] \times \{b m^{-k}\}, \ \  j,k\in \mathbb{N},\ 0\leq a\leq n^{j}-1, \ 0\leq b\leq m^{k}.$$Then for $q\geq j$ there are $n^{q-j}$ level-$q$ rectangles $R_{i_1,\ldots,i_q}$ that touch $S$ on each side. If $S\subset F$ then one or the other of every such abutting pair of level-$q$ rectangles must be selected, and the probability of this is $(1-(1-p)^2)^{n^{q-j}}=(2p-p^2)^{n^{q-j}} \to 0$ as $q\to \infty$. Thus $\P(S \subset F)=0$ for all such $S$, and there are countably many segments $S$ of this form.

Now let $j \geq 1$ and $0\leq a \leq m^{j}-1$; we consider rows of level-$q$ rectangles in the column  $[an^{-j}, (a+1)n^{-j}]\times [0,1]$. For $q\geq j$ the row 
\be\label{rowrect}
[cn^{-q}, (c+1)n^{-q}] \times [bm^{-q}, (b+1)m^{-q}], \ \ an^{q-j}\leq c\leq (a+1)n^{q-j}-1
\ee
contains $n^{q-j}$ level-$q$ rectangles each of which is selected independently with probability $p$. The probability of all the rectangles in this row being selected is $p^{n^{q-j}}$. There are $m^q$ such rows of the form \eqref{rowrect} with  $0\leq b\leq m^q -1$ which are independent of each other, so the probability of at least one such full row being selected is $1- (1-p^{n^{q-j}})^{m^q}\to 0$ as $q\to\infty$, since $m^q \log(1-p^{n^{q-j}})  \sim -m^q p^{n^{q-j}}\to 0$. 

A line segment  
$[an^{-j}, (a+1)n^{-j}] \times \{y\}$ in $F$ where $y$ is not of the form $b m^{-k}$ must lie in 
some row of selected rectangles of the form \eqref{rowrect} for every $q\geq j$, but almost surely there is no such sequence of rows of selected rectangles.

Thus almost surely, $F$ contains no line segment of the form  $[an^{-j}, (a+1)n^{-j}] \times \{y\}$ for each, and thus for all, $j$ and $a$. Every horizontal line segment contains some interval of this form, so $F$ contains no horizontal line segments. The argument for $F_k$ is similar, just considering $q\geq k$.
\end{proof}

We need the following intuitively obvious result on intersecting connected components of a compact set.

\begin{lemma}\label{con}
Let $K$ be a non-degenerate compact convex subset of $\mathbb{R}^2$ with perimeter curve $\partial K$. Let ${\bf x}_1,{\bf x}_2,{\bf x}_3,{\bf x}_4 $ be (not necessarily distinct) points in clockwise order on  $\partial K$. Let $F$ be a compact subset of $K$. If ${\bf x}_1$ and ${\bf x}_3$ are connected in $F$ and ${\bf x}_2$ and ${\bf x}_4$ are connected in $F$ then ${\bf x}_1,{\bf x}_2,{\bf x}_3,{\bf x}_4 $ are all connected in $F$, i.e. are all in the same connected component of $F$.
\end{lemma}
\begin{proof}
If two of the ${\bf x}_i$ coincide then the conclusion follows since connected components that intersect in a point 
must belong to a single connected component.

Otherwise, suppose ${\bf x}_1,{\bf x}_3$ are in a connected component $A$ of $F$ and  ${\bf x}_2,{\bf x}_4$ are in a connected component $B$. Then $A,B$ are compact, so if they are disjoint they are separated by a positive distance $d$. A compact set $E$ is connected if and only if for all $\epsilon >0$ it is $\epsilon$-connected, that is for all ${\bf y}_0, {\bf y}\in E$ there is an $\epsilon$-chain of points ${\bf y}_0,{\bf y}_1,\ldots,{\bf y}_n = {\bf y}$ in $E$ such that $|{\bf y}_i-{\bf y}_{i+1}|< \epsilon$ for all $0\leq i\leq n-1$, see \cite{New}. Thus we may find a $d/3$-chain in $A$ joining ${\bf x}_1$ and ${\bf x}_3$ and one in $B$ joining ${\bf x}_2$ and ${\bf x}_4$. The polygonal curves defined by these chains must cross at some point ${\bf w}\in K$, so there are points in both $A$ and $B$ within distance $d/3$ of ${\bf w}$, so $A$ and $B$ have separation at most $2d/3$, a contradiction. Hence $A$ and $B$ are not disjoint, so are in the same connected component of $F$.
\end{proof}

We now show that if there is a positive probability of $F$ crossing horizontally a (perhaps narrow) column, we can find a `link' component of $F$ of a particular form crossing the column, and then join up several such links to obtain a crossing of $[0,1]^2$. This extends a technique used in \cite{DM}.

For $k\in \mathbb{N}$ and $i\in \mathbb{Z}$ we write $I^k_i$ for the interval
\be\label{Iki}
 I^k_i := [ im^{-k},  (i+1)m^{-k}] \subset\mathbb{R}.
 \ee
 For $0\leq x\leq 1$ let $L_x$ denote the vertical line-segment $\{x\}\times [0,1]$.

\begin{lemma}\label{Col} Let $0<p\leq 1$ and let $q\geq1$ be an integer. Suppose that there is positive probability that $F$ includes a horizontal crossing of the column $[0,n^{-q}] \times [0,1]$. Then there exist an integer $k >q$, intervals $ I^k_i,  I^k_j, I^k_{i'},  I^k_{j'}$ with $|i-i'|\geq 2$ and  $|j-j'|\geq 2$ , and integers $1\leq a < \frac{1}{2}n^{k-q}<b<n^{k-q}$   such that with positive probability there is a connected component $F_0$ of $F$ such that:

$(i)$  $F_0$ intersects the line segments $L:=\{0\}\times I^k_i,\, L':=\{an^{-k}\}\times I^k_{i'}, \, R':=\{bn^{-k}\}\times I^k_{j'}, \, R:=\{n^{-q}\}\times I^k_{j}$; 

$(ii)$  $L$ is connected to $\L'$ in 
$F_0 \cap ([0,an^{-k}] \times [0,1])$;  

$(iii)$  $R'$ is connected to $R$ in 
$F_0\cap ([bn^{-k},n^{-q}] \times [0,1])$;

\noindent   see Figure \ref{narrow}.  
\end{lemma}
\begin{figure}[h]
\begin{center}
\includegraphics[width = 0.5\textwidth]{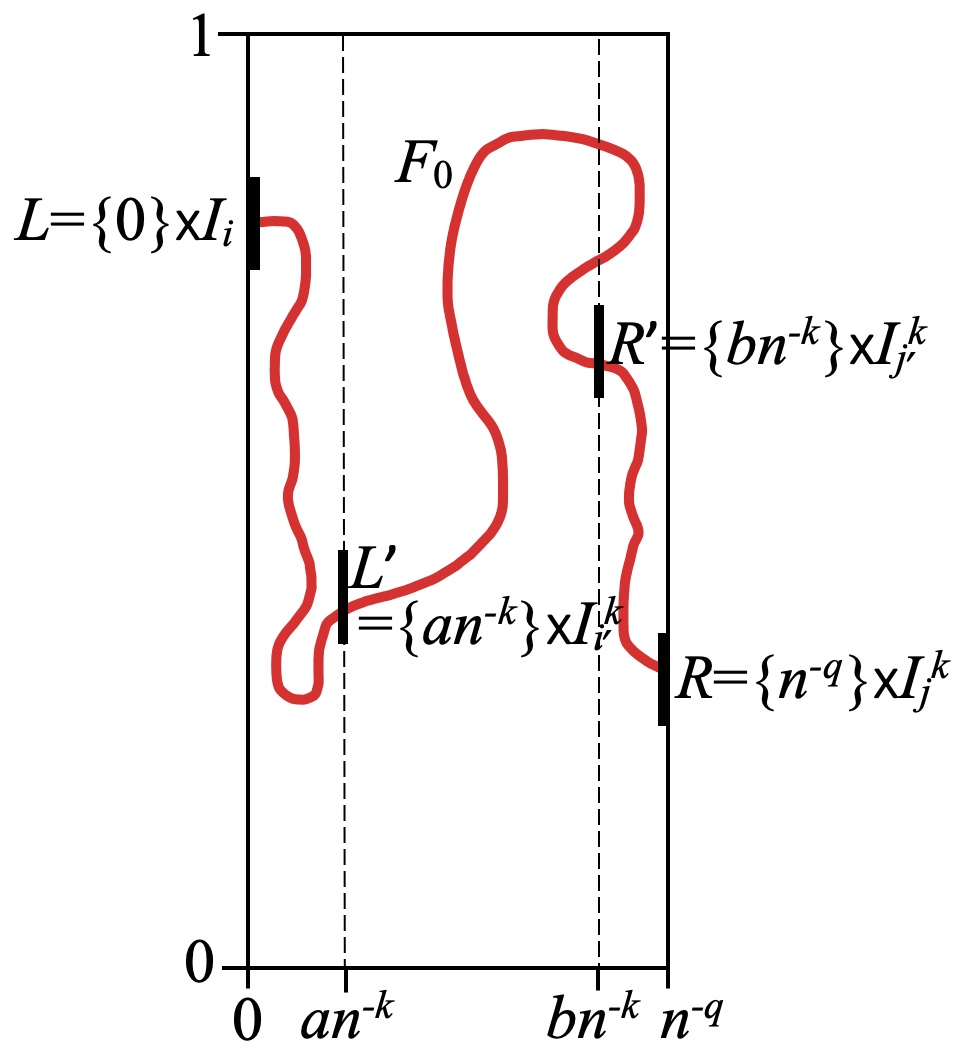}
\caption{Line segments $L,L',R',R$ connected by a component $F_0$ (represented here by a curve for illustrative convenience)}\label{narrow}
\end{center}
\end{figure}
\begin{proof}
If $F$ is a compact set that includes an H-crossing of $[0,n^{-q}] \times [0,1]$, there is a connected component $F_0\subset F\cap ([0,n^{-q}] \times [0,1])$ that gives a H-crossing of $[0,n^{-q}] \times [0,1]$.

Let $y\in [0,1]$ be such that the point $(0,y) \in F_0$. By Lemma \ref{seg} the segment $[0,\frac{1}{2}n^{-q}]\times \{y\}$ is almost surely not contained in the compact set $F$, so as the points $\{an^{-k}\}_{k,a}$ are dense in $[0,1]$ there are almost surely $k,a \in \mathbb{N}$ with $0< an^{-k}<\frac{1}{2}n^{-q}$ such that $(an^{-k},y)\notin F$. Since  $L_{an^{-k}}$ must intersect $F_0$ we may find $y'\in [0,1]$ with $(an^{-k},y') \in F_0$ and $y'\neq y$ such that $(an^{-k},y')$ is connected to $(0,y)$ in $F_0\cap([0,an^{-k}]\times [0,1])$. There are intervals $ I^k_i, I^k_{i'}$ with $0\leq i,i'\leq m^k -1$ and $(0,y) \in \{0\} \times I^k_i $ and $(an^{-k},y')\in \{an^{-k}\} \times I^k_{i'} $. If $|i-i'| <2$ then we may replace $k$ with a larger $k'$ and find smaller intervals $ I^{k'}_{i''} \ni y \text{ and } I^{k'}_{i'''}\ni y'$ with $|i''-i'''| \geq 2$ and $a' = an^{k'-k}$ which satisfy ($ii$).

In the same way, working near the right-hand side of the rectangle $[0,n^{-q}] \times [0,1]$, we may obtain intervals  $I^k_{j}$ and $I^k_{j'}$ and an integer $b$ satisfying ($iii$). Finally, to ensure a common value of $k$ in ($ii$) and ($iii$), we take the larger of the values of $k$ obtained in ($ii$) and ($iii$) and reduce the lengths of the intervals  $ I^k_i,  I^k_{i'}, I^k_j,   I^k_{j'}$ as appropriate.

Thus if there is a positive probability that the random set $F$ provides an H-crossing of $[0,n^{-q}] \times [0,1]$ then there are integers $k, a, b$ and intervals $ I^k_i,  I^k_{i'},  I^k_{j'}, I^k_j$ for which a connected   $F_0 \subset F$ satisfies (i)-(iii). But there are countably many possibilities for these integers and intervals, so at least one combination occurs with positive probability.
\end{proof}

We proceed as in \cite[Proof of Lemma 5.1]{DM} to link together columns of the form considered in Lemma \ref{Col} to obtain a crossing of the unit square.

\begin{prop}\label{verthor} Let $0<p\leq 1$ and let $[r,s]\subset [0,1]$ be an interval. If there is a positive probability that $F$ gives an H-crossing of the column $[r,s]\times [0,1]$ then $\theta_H(p)>0$.  Similarly, if there is a positive probability that $F$ provides a V-crossing of the rectangle $[0,1]\times [r,s]$ then $\theta_V(p)>0$.
\end{prop}
\begin{figure}[h]
\begin{center}
\includegraphics[width = 0.65\textwidth]{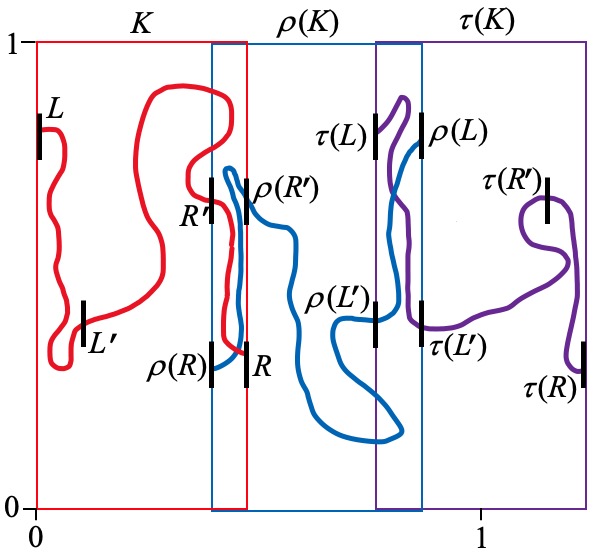}
\caption{Constructing a  horizontal crossing of $[0,1]^2$ given  crossings of vertical strips.}\label{wide}
\end{center}
\end{figure}
\begin{proof}
Take $q\in \mathbb{N}$ such that $n^{-q}< \frac{1}{2} (r-s)$ and an integer $0\leq u \leq n^q -1$ such that $[u n^{-q}, (u+1)n^{-q}] \subset [r,s]$. Then  with positive probability $[u n^{-q}, (u+1)n^{-q}]\times [0,1]$  is H-crossed by $F$ and thus by $F_q$. By statistical self-affinity the sets  $F_q \cap ([u n^{-q}, (u+1)n^{-q}]\times [0,1])$ and 
$F_q \cap ([0, n^{-q}]\times [0,1])$ have the same distribution modulo a horizontal shift of $u n^{-q}$, so in particular there is  positive probability that $F_q$ includes an H-crossing of the rectangle $K:=[0, n^{-q}]\times [0,1]$.

Lemma \ref{Col} provides an integer $k>q$ and specific segments $L,L',R',R$  which with probability $p_0>0$, say, are connected as in ($i$)-($iii$) in $F$ and thus in $F_k$ .  Write $C(K;L,L',R',R)$ for the event that $F_k$ has a component that connects the vertical segments $L,L',R',R$ within the rectangle $K$, connecting ${L}$ and $L'$ within the vertical strip with sides containing  ${L}$ and $L'$, and connecting $R'$ and $R$ within the vertical strip with sides containing  $R'$ and $R$. Let $\rho$ denote the operation on $\mathbb{R}^2$ of reflection in the vertical line $x= \frac12 (bn^{-k}+n^{-q})$, and let $\tau$ denote horizontal translation by $(b-a +n^{k-q})n^{-k}$. These mapping have been set so that the pair of right-hand intervals in the following chain of events align vertically with the pair of left-hand intervals of the succeeding event. 

We define a sequence of events:
\begin{equation}\label{events}
\begin{array}{ll}
&C(K;L,L',R',R),\; C\big(\rho(K);\rho(R),\rho(R'),\rho(L'),\rho(L)\big),
 \\

&C\big(\tau(K);\tau(L),\tau(L'),\tau(R'),\tau(R)\big),\;C\big(\tau\rho(K);\tau\rho(R),\tau\rho(R'),\tau\rho(L'),\tau\rho(L)\big), 
\\

&C\big(\tau^2(K);\tau^2(L),\tau^2(L'),\tau^2(R'),\tau^2(R)\big),\; C\big(\tau^2\rho(K);\tau^2\rho(R),\tau^2\rho(R'),\tau^2\rho(L'),\tau^2\rho(L)\big), \\

&C\big(\tau^3(K);\tau^3(L),\tau^3(L'),\tau^3(R'),\tau^3(R)\big), \ldots,
\end{array}
\end{equation}

\noindent continuing until either $\tau^\ell(K)$ or $\tau^\ell\rho(K)$ overlaps the right-hand edge of $[0,1]^2$ for some $\ell \in \mathbb{N}$  in which case we ignore the level-$k$ rectangles outside the square.

Note that if consecutive events in this list occur then the intervals in both events lie in the same component of  $F_k$. For example, for $p\geq 0$,
$\tau^p(R')$ and $\tau^p\rho(R)$ lie in a common vertical line as do $\tau^p(R)$  and
$\tau^p\rho(R')$. Since $\tau^p(R)$ is joined to $\tau^p(R')$ and $\tau^p\rho(R)$ is joined to $\tau^p\rho(R')$ between these two vertical lines, Lemma \ref{con} implies that all four segments intersect a common component of $F_k$, and thus all the segments listed in these two events are joined. In particular, if all the events in this list occur, then there is a connected component of $F_k$ that intersects $L$ and $\tau^\ell(R)$ or $\tau^\ell\rho(L)$ and so provides an H-crossing of $[0,1]^2$. 

But each event in the list has the same distribution as $C(K;L,L',R',R)$, modulo a translation or reflection, so occurs with probability $p_0$. These are all increasing events, so by Harris' inequality the probability of them all occurring is at least $p_0^{2\ell +1}>0$, in which case there is a positive probability of an H-crossing in $F_k$ and thus in $F$ (since there is positive probability of all the rectangles up to level-$k$ all being selected).

The argument for vertical crossings is similar.
\end{proof}

\noindent {\it Proof of Theorem \ref{thmbrief}}
For  $k\in \mathbb{N}$ and $0\leq b \leq m^k-1$ let $A(k,b)$ be the event that $F$ includes a vertical crossing of the horizontal strip $[0,1]\times [bm^{-k},(b+1)m^{-k}]$. Then
$$\{F\text{ has an H-crossing of } [0,1]^2 \} \subset \bigcup_{k=1}^\infty \bigcup_{b=0}^{m^k-1}A(k,b) \cup\{ F \text{ contains a horizontal line}\}.$$
Thus if $\P\{F \text{ has an H-crossing of }[0,1]^2\}>0$ then $\P(A(k,b)) >0$ for some $k,b$, since $\P\{ F \text{ contains a horizontal line}\} =0$ by Lemma \ref{seg}. Thus $F$ includes a V-crossing of $[0,1]\times [bm^{-k},(b+1)m^{-k}]$ for some $k,b$, so $F$ has a V-crossing of $[0,1]^2$ with positive probability by Proposition \ref{verthor}.

Hence if $\theta_H(p)>0$ then $\theta_V(p)>0$, with the reverse inequality following using the `vertical' part of Proposition \ref{verthor}.
$\Box$ \hfill

\subsection{Proof of Theorems \ref{phase} and \ref{cty}}

Now we write $p_C:= p_H=p_V$.

For Theorem \ref{phase} we must show that $\theta_H(p_C)>0$ and $\theta_V(p_C)>0$,which we do by adapting the proof in \cite[Section 5]{DM}. It seems awkward to show this directly, but rather we first consider crossings in the random model starting with the rectangle $[0,1]\times [0,2]$ and proceeding as before, with subdivision into a hierarchy of $n^{-k}\times m^{-k}$ rectangles for $k\geq 1$, with each rectangle selected with probability $p$. In other words there are two vertically adjacent independent copies of the original process in $[0,1]\times [0,1]$ and $[0,1]\times [1,2]$ with connection across $[0,1]\times \{1\}$ permitted in the natural way. We write $\tau_H(p)$ for the probability that there is a horizontal crossing of $[0,1]\times [0,2]$. Similarly, we may consider percolation on the rectangle $[0,2]\times [0,1]$, that is two horizontally adjacent copies of our original model, and write $\tau_V(p)$ for the probability of a vertical crossing of $[0,2]\times [0,1]$.

\begin{lemma}\label{double}
Let $0<p\leq 1$. The probabilities  $\theta_H(p), \theta_V(p),\tau_H(p),\tau_V(p)$ are either all $0$ or all strictly positive.
\end{lemma} 
\begin{proof}
Clearly $\tau_H(p)=0$ implies $\theta_H(p)=0$. To show that if $\tau_H(p)>0$ then $\theta_H(p)>0$ is a variant of  \cite[Proof of Lemma 5.1]{DM}. More specifically, this is similar to the construction of Lemma \ref{Col} but starting with a positive probability H-crossing of the rectangle $[0,1]\times [0,2]$ and extending a connected component horizontally in stages by chaining together links as in the proof of Proposition \ref{verthor} to obtain a positive probability crossing in $F_k$ for some $k$, and thus in $F$, of $[0,n]\times [0,2]$, considered as a $n\times 2$ array of independent copies of the random process on $[0,1]\times [0,1]$. But the probability of an H-crossing of $[0,n]\times [0,2]$ in $F$ is, by self-affinity, the same as the probability of an H-crossing of $[0,1]\times [0,2m^{-1}]\subset [0,1]\times [0,1]$ in $F_2$ which is therefore positive giving a positive probability of an H-crossing in $F$.

Thus $\tau_H(p)>0$ if and only if  $\theta_H(p)>0$. Similarly $\tau_V(p)>0$ if and only if  $\theta_V(p)>0$. By Theorem \ref{thmbrief} $\theta_H(p)>0$ if and only if $\theta_V(p)>0$, so the conclusion follows.
\end{proof}

The next lemma shows that for a given $p$,  $\tau_H(p)$ and $\tau_V(p)$ cannot both be positive but very small.

\begin{lemma}\label{tau}
Let $0<p\leq 1$. If $\theta_H(p)\geq \theta_V(p)$ then either $\tau_H(p)= 0$ or $\tau_H(p)\geq  (4m)^{-n/(n-1)}$.
If $\theta_V(p)\geq \theta_H(p)$ then either $\tau_V(p)= 0$ or $\tau_V(p)\geq  (4n)^{-m/(m-1)}$.
\end{lemma} 
\begin{proof}
Assume that $\theta_H(p)\geq \theta_V(p)$; the reverse case is similar. We adapt the proof of \cite[Theorem 5.4]{DM}.

Consider one of the level-1 columns of $[0,1]\times [0,2]$ comprising $2m$ size $m^{-1}\times n^{-1}$ rectangles, $Q_1,\ldots,Q_{2m}$. If there is a horizontal crossing of $[0,1]\times [0,2]$ and thus of this column, then at least one of the following events occurs:
$$V^p := \bigcup_{i=1}^{2m}\{ \text{there exists a vertical crossing of } Q_i\} \text{ and }$$
$$H^p := \bigcup_{i=1}^{2m-1}\{ \text{there exists a horizontal crossing of } (Q_i\cup Q_{i+1})\}.$$
Using the statistical self-similarity between the process in $[0,1]^2$ and each $Q_i$ and between that in $[0,1]\times [0,2]$ and each pair $Q_i\cup Q_{i+1}$, allowing for the probability of these rectangles not being deleted at level-1, and using that there are $n$ such columns, we relate the probabilities:
\begin{eqnarray*}
\tau_H(p)&\leq &(\P(V^p\cup H^p))^n \leq \big(\P(V^p)+ \P(H^p)\big)^n\\
&\leq &\Big(2mp\,\theta_V(p) +(2m-1)\big(p^2\tau_H(p) +2p(1-p) \theta_H(p)\big)\Big)^n\\
&\leq &\Big(2mp +(2m-1)\big(p^2 +2p(1-p)\big)\Big)^n\tau_H(p)^n\\
&\leq &\big(2m +(2m-1)\big)^n\tau_H(p)^n\\
&\leq& (4m)^n\tau_H(p)^n,
\end{eqnarray*}
using that $\theta_V(p)\leq \theta_H(p)\leq \tau_H(p)$. Thus $\tau_H(p)= 0$ or $\tau_H(p)\geq  (4m)^{-n/(n-1)}$.

The case of $\theta_V(p)\geq \theta_H(p)$ is similar, considering crossings of $[0,2]\times [0,1]$.
\end{proof}

Combining the previous two lemmas yields Theorem \ref{phase}.

\noindent {\it Proof of Theorem \ref{phase}}
First we note that $\theta_H(p)$ is increasing and upper-semicontinuous and therefore right-continuous in $p$. To see this, the probability $\theta^k_H(p)$ of a horizontal crossing of  $[0,1]^2$ within $E_k$ (see \eqref{Ek}) is a polynomial in $p$, since the probability of any given configuration of rectangles at the first $k$ levels is a polynomial of the form $p^\alpha(1-p)^\beta$, and summing  the disjoint probabilities of the configurations that include a crossing within $E_k$ gives a polynomial. Thus $\theta^k_H(p)$ is increasing and continuous in $p$ for each $k$. As $k\to\infty$,  $\theta^k_H(p) \searrow\theta_H(p)$, recalling that there is a crossing in $F$ if and only if there is a crossing in $E_k$ for all $k$, so $\theta_H(p)$ is increasing and upper-semicontinuous, so right-continuous, as the infimum of continuous functions.

Recall that $p_C$ is the common critical probability for $\theta_H(p)$ and $\theta_V(p)$. There is a strictly decreasing sequence $p_n\searrow p_C$ either with $\theta_H(p_n) \geq\theta_V(p_n)$ for all $n$ or with $\theta_H(p_n) \leq\theta_V(p_n)$ for all $n$, suppose the former, the latter case is similar. By  Lemma \ref{tau}, $\tau_H(p_n)\geq  (4m)^{-n/(n-1)}$ for all $n$, so since $\tau_H(p)$ is right-continuous, $\tau_H(p_C)\geq  (4m)^{-n/(n-1)}$. By Lemma \ref{double} $\theta_H(p_C)>0$ and  $\theta_V(p_C)>0$. 
\hfill $\Box$

We now establish Theorem \ref{cty} reflecting the nature of $F$ when $p<p_C$ and when $p\geq p_C$.

\noindent {\it Proof of Theorem \ref{cty}}

\noindent (i) Let $1/mn<p<p_C$ so conditional on $F\neq \emptyset$, $\dim_H F >0$ by Proposition \ref{dims}, so $F$ is uncountable.

Suppose there is a positive probability of a non-trivial connected component $F_0$ of $F$. Then  there are  $a,k \in \mathbb{N}$ such that either $F$ includes an H-crossing of $[an^{-k}, (a+1)n^{-k}] \times [0,1]$ or a V-crossing of  $[0,1]\times [an^{-k}, (a+1)n^{-k}]$; by countable decomposition we may select an orientation and  specific $a$ and $k$ for which there is a positive probability of this occurring.
By Proposition \ref{verthor} there is a positive probability of either a V-crossing or an H-crossing of $[0,1]\times [0,1]$, implying that $\theta_H(p),\theta_V(p)>0$, contradicting that $p<p_C$.

\noindent (ii) Let $p_C\leq p<1$. We call a rectangle  $R$ in $E_k$ an {\it island} if $R\cap F$
contains a non-trivial connected component and has an empty perimeter band (i.e. the distance from $R\cap F$ to the perimeter of $R$ is positive). 
 There is a probability $p_0>0$ that $E_0$ is an island, since there is a positive probability that none of the level-2 perimeter rectangles are selected but some interior sub-rectangle is H-crossed by $F$, so by statistical self-affinity there is a probability $p_0$ that each $R$ in $E_k$ is an island. In particular each island in $E_k$ contains a non-trivial connected component which is disjoint from all other level-$k$ rectangles. 

Given $M\in \mathbb{N}$ we may proceed with the step by step construction of $F$ until we stop at the first  $k$ such that either $|E_k| \geq M $ or $E_k =\emptyset$, the former will occur with probability at least $p_N$, the probability of non-extinction of the construction. Conditioning on $|E_k| \geq M $, the probability that at least $\frac12  p_0M$ of these  $R$ are islands is at least $1- 4/(Mp_0)$ using Chebychev's inequality. By taking $M$ large, it follows that the number of islands, and thus the number of disjoint non-trivial connected components of $F$ is almost surely unbounded conditional on non-extinction. 

Every non-trivial connected component crosses some horizontal or vertical strip of rectangles and each such strip is almost surely crossed by finitely many components, just as in the statistically self-similar case, see \cite[Theorem 4.1]{Mee}. Thus the number of connected components is countable.
\hfill $\Box$

\section{Final remarks}\label{final}
Once it is established that the critical probabilities for H-crossings and V-Crossings are equal, many other topological properties of the statistically self-affine sets $F$ follow in a very similar  way as for statistically self-affine sets.

Path and arc connectivity (which are equivalent in $\mathbb{R}^n$) are stronger conditions than topological connectivity. Nevertheless, the proof of \cite[Theorem 3.1]{Mee} adapts to show that 
if  $p_C\leq p \leq 1$ then with positive probability $F$ is crossed both horizontally and vertically by a path or arc; thus the critical probability for path/arc connectivity  is the same as for topological connectivity. 

Similarly, adapting another proof of \cite[Lemma 2.4]{Mee} shows that almost surely just finitely many disjoint connected components of $F$ cross $F$ horizontally and finitely many cross vertically.

We end with some questions. Because in the $k$th level rectangular grid there are far more rectangles in each column than in each row, one would expect that the chance of an $V$-crossing would be much smaller than that of an  H-crossing. However, this does not seem obvious.

\begin{ques}\label{Q1}
 Is $\theta_V(p)\leq \theta_H(p)$ for all $m>n\geq 2$ and all $p_C\leq p\leq 1$? Less strongly, is there a constant $c>0$ such that $\theta_V(p)\leq c\, \theta_H(p)$ for all $p_C\leq p\leq 1$?
 \end{ques}
 An affirmative answer to Question \ref{Q1} would remove the need to consider two cases in the proof of Theorem \ref{phase}. 
\begin{ques}\label{Q2}
 Are $\theta_H(p)$ and $\theta_V(p)$ continuous on the interval $[p_C,1]$?
 \end{ques}
 This does not seem to be known even for the statistically self-similar $m\times m$ case.

\section*{Acknowledgements}
The authors are grateful to the referees for their helpful comments.

\medskip

\noindent Email addresses: \texttt{kjf@st-andrews.ac.uk}, \texttt{tf66@st-andrews.ac.uk}

\end{document}